\documentclass[preprint,12pt,authoryear]{elsarticle}
\usepackage{amsmath}
\usepackage{amsfonts}
\usepackage{amssymb}
\usepackage{amsthm}
\usepackage{hyperref}
\usepackage{url}
\usepackage{indentfirst}
\usepackage{tikz}
\usepackage{filecontents}
\usepackage{booktabs}
\usepackage[noend]{algpseudocode}
\usepackage[margin=1.5in]{geometry}    
\usepackage{algorithm}
\usepackage{graphicx}
\usepackage{float}
\usetikzlibrary{shapes,positioning}
\usepackage{enumerate}
\usepackage{comment}
\usepackage{caption}

\usepackage{subcaption}

\usepackage{natbib}

\newtheorem{proposition}{Proposition}[section]

\newtheorem{example}{Example}[section]

\usepackage{chngcntr}
\counterwithout{remark}{section}
\counterwithout{figure}{section}
\counterwithout{table}{section}
\counterwithout{theorem}{section}
\counterwithout{lemma}{section}
\counterwithout{proposition}{section}
\begin{document}
\begin{frontmatter}

\title{Bridge distances for networks}

\author[hse,rutgers]{Vladimir Gurvich}
\ead{vgurvich@hse.ru}

\author[rutgers_business]{Mariya Naumova}
\ead{mnaumova@business.rutgers.edu}

\address[hse]{National Research University Higher School of Economics, Moscow, Russia}
\address[rutgers]{Rutgers Center for Operations Research, Rutgers University, Piscataway, New Jersey, United States}
\address[rutgers_business]{Rutgers Business School, Rutgers University, Piscataway, New Jersey, United States}

\begin{abstract}
Let  $G = (V,E)$ be a finite directed graph  
with a non-negative real length $\mu_e$  assigned 
to every directed edge $e \in E$. 
We assume that $\mu_e = +\infty$ for every non-edge $e \not\in E$.   
Fix any two distinct vertices  $a, b \in V$. 
A directed path from $a$ to $b$  is called an $(a,b)$-path. 
An edge $e$  is called an $(a,b)$-bridge if it belongs to all  $(a,b)$-paths. 
Furthermore, it is not difficult to show that  
all $(a,b)$-paths pass all $(a,b)$-bridges in the same order.  
Define the distance $\mu(a,b)$  from  $a$  to  $b$  
as the sum of lengths of all  $(a,b)$-bridges. 
Furthermore, $\mu(a,b) = \infty$ if there are no  $(a,b)$-paths and 
$\mu(a,b) = 0$  if  $(a,b)$-paths exist   
but there are no  $(a,b)$-bridges.  
It is easily seen that  $\mu(a,b)$  can be computed in polynomial time
and the metric inequality 
$\mu(a,b) \leq \mu(a,c) + \mu(c,b)$ holds for every $a,b,c \in V$. 
Furthermore, equality holds if and only if 
each  $(a,b)$-bridge is either an $(a,c)$- or a  $(c,b)$-bridge. 
\newline
We will show that this is  a special limit case $r=s \rightarrow 0$ 
of the inequality $\mu(a,b)^{s/r} \leq \mu(a,c)^{s/r} + \mu(c,b)^{s/r}$ 
obtained for all positive real parameters  $r$ and $s$  in the paper 
``Metric and ultrametric inequalities for directed graphs'', 
Discrete Appl. Math. 314 (2022) 93--104, 
along with 3 other limit cases 
 $r=s \rightarrow \infty$,  $r=1, s \rightarrow \infty$, and   $s = 1, r \rightarrow 0$,  
 considered in that paper.
\end{abstract}


\end{frontmatter}

\section{Introduction}
\label{s0}
The present article is an extension of \cite{Gur22}.  

We assume that the reader is familiar with the main results and notation of \cite{Gur22}, 
although it is not necessary to read the whole paper carefully.  

\subsection{Main results of \cite{Gur22}}
Let  $G = (V,E)$ be a finite weighted directed graph (digraph)  
with a non-negative real weight 
(which will be called {\em length})  $\mu_e$  assigned 
to every directed edge $e \in E$.  
Loops are forbidden, 
but parallel edges, directed oppositely or similarly, are allowed.  
The values  $\mu_e = 0$ and $\mu_e = \infty$ are also allowed. 
In particular, we can extend  $G$  to a complete digraph 
assigning  $\mu_e = \infty$ to every nonedge $e \not\in E$. 
We introduce $\lambda_e = 1/\mu_e$  assuming standardly that 
$1/0 = \infty$ and $1/\infty = 0$. 

In \cite{Gur22}, we fix two positive real parameters $r$  and $s$ 
and for each ordered pair of vertices $a,b \in V$, 
we define the distance $\mu_{a,b}$  from  $a$  to  $b$. 
We use the following standard names: 
$a$ - source, $b$ - sink, pair $a,b$ - poles, 
the network with fixed source and sink 
is called the {\em two-pole} network, 
$\mu_{a,b}$ - its resistance or length, 
$\lambda_{a,b} = 1/\mu_{a,b}$ - 
the conductance or the inverse length. 

As the main result, for every triple $a,b,c \in V$ 
we prove the inequality 

\begin{equation} 
\label{eq-main} 
\mu(a,b)^{s/r} \leq \mu(a,c)^{s/r} + \mu(c,b)^{s/r}.
\end{equation}

It implies the standard metric inequality 
$\mu(a,b) \leq \mu(a,c) + \mu(c,b)$  if $s \geq r$ and 
the ultrametric inequality 
$\mu(a,b) \leq \max(\mu(a,c),\mu(c,b))$ in the limit case $r/s \rightarrow 0$. 

The equality in \ref{eq-main}  holds if and only if 
every directed path from $a$ to $b$ (so-called $(a,b)$-path) contains  $c$. 

Note that this claim holds for every proper case 
($r>0$ and $s>0$), 
but in some limit cases the only if part fails; see \cite{Gur22}. 

\subsection{Special cases and history} 

{\bf Isotropic networks.}
A digraph  $G$  is called {\em symmetric} if its edges
are split into a set of oppositely directed pairs:
$e' = (v', v''), e'' = (v'', v')$.
Respectively, a network is called {\em symmetric} or {\em isotropic} 
if its graph is  symmetric and
$\mu_{e'} = \mu_{e''}$  for each pair  $e', e''$  considered above.
Then, one can replace each such pair  $e', e''$
in  $G$  by a non-directed edge  $e$, 
thus, replacing the digraph of the network by a non-directed graph.
Set  $\mu_e  = \mu_{e'} = \mu_{e''}$  for every such  $e$. 

For this case, inequality (\ref{eq-main}) was proven in 1987. 
Note that the isotropic case is simpler than 
the anisotropic one due to uniqueness of the potentials; 
see \cite{Gur22}  for more details. 
Equality $\mu_{a,b} = \mu_{b,a}$ 
holds for any pair of poles  $a$  and  $b$  of an isotropic network.  
Thus, symmetric networks generate metric or ultrametric spaces, 
while we have only quasi-metric or quasi-ultrametric spaces  
for anisotropic networks.

\bigskip

{\bf Linear networks}, $r=s=1$.
In the isotropic linear case, the triangle inequality for resistances 
was discovered by Gerald Subak-Sharpe  (\cite{Sha67,Sha67a}). 
This result was rediscovered several times later.
For the linear anisotropic networks   
inequality  (\ref{eq-main})
(along with many related results) was  proven in \cite{YSL16}.

\section{Two proper and four limit cases}

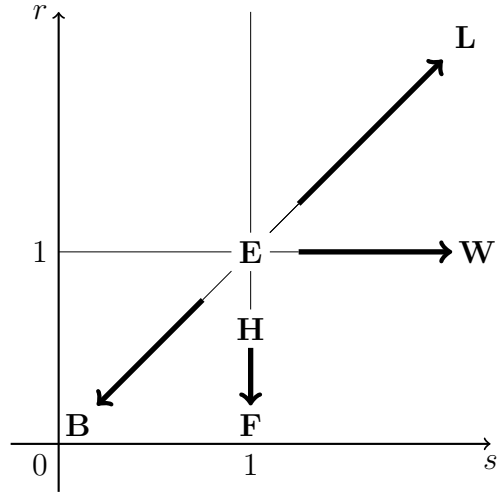
\begin{figure}[h]
\begin{center}
\begin{tikzpicture}
  \draw [->,thick] (-.25in,0in) to (2.25in,0in);
  \draw [->,thick] (0,-0.25in) to (0,2.25in);
  
  \draw [very thin] (0,1in) to (.9in,1in);
  \draw [very thin] (1.1in,1.1in) to (1.25in,1.25in);
  \draw [very thin] (1.1in,1in) to (1.25in,1in);
  
  \draw [very thin] (1in,.7in) to (1in,.9in);
  \draw [very thin] (1in,1.1in) to (1in,2.25in);
  
  \draw [very thin] (0.75in,0.75in) to (.9in,.9in);
  \draw [very thin] (1.1in,1.1in) to (1.25in,1.25in);
  \draw [->,line width=2pt] (0.75in,0.75in) to (0.2in,0.2in); 
  \draw [->,line width=2pt] (1in,.5in) to (1in,0.2in);
   
  \draw [->,line width=2pt] (1.25in,1in) to (2.05in,1in);
  \node at (2.18in,1in) {$\bf{W}$};
  
  \draw [->,line width=2pt] (1.25in,1.25in) to (2in,2in);
  \node at (2.12in,2.12in) {$\bf{L}$};
  
  \node at (1in,1in) {$\bf{E}$};
  \node at (1in,0.6in) {$\bf{H}$};
  \node at (0.1in,0.1in) {$\bf{B}$};
  \node at (1in,0.1in) {$\bf{F}$};
  
  \node at (0,1in) [shape=coordinate,label=left:$1$]  {};
  \node at (1in,0) [shape=coordinate,label=below:$1$]  {};
  \node at (0,2.25in) [shape=coordinate,label=left:$r$]  {};
  \node at (2.25in,0) [shape=coordinate,label=below:$s$]  {};
  \node at (0,0) [shape=coordinate,label={below left:$0$}] {};
\end{tikzpicture}
\end{center}
  \caption{The $s \times r$ diagram} 

  \label{pic1}
\end{figure}

In Figure \ref{pic1}  we have two proper cases 
\begin{itemize}
\item[] $r=s=1$  (E, electricity, resistance), 
\item[]$r=1/2, s=1$ (H, hydraulics and gas dynamics, resistance), 
\end{itemize}
and four limit cases
\begin{itemize}
\item[] $r=s \rightarrow \infty$ (L - length; the shortest  (a,b)-path length or travel time), 
\item[] $r=1, s \rightarrow \infty$ (W - width, the widest bottleneck (a,b)-path width), 
\item[] $r\rightarrow 0, s=1$ (F - flow, the inverse  $(a,b)$-capacity),  
\item[] $r=s \rightarrow 0$  (B - bridges, the total length of the  $(a,b)$-bridges).  
\end{itemize}

Let us briefly recall the interpretation of the five cases from \cite{Gur22}. 
Case B, which was not addressed in \cite{Gur22}, will be considered in the next section.  

We will analyze two simplest two-pole networks in Figure \ref{pic2} 
and for each of them give a formula for $\mu_{a,b}$ 
(or, equivalently, for $\lambda_{a,b}$; recall that $\mu = 1/\lambda$). 

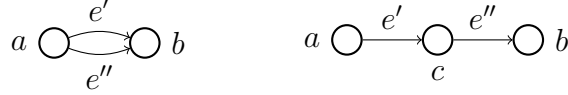
\begin{figure}[h]
\begin{center}

\parbox[c]{1.2in}{
\begin{tikzpicture}[scale=.6]
  \node at (-1,0) [shape=circle,minimum size=6pt,draw=black,thick,label=left:$a$] (a) {};
  \node at (1,0) [shape=circle,minimum size=6pt,draw=black,thick,label=right:$b$] (b) {};
  \draw[->] (a) to [out=20,in=160] node [above] {$e'$} (b);
  \draw[->] (a) to [out=-20,in=-160] node [below] {$e''$} (b);
\end{tikzpicture}
}\qquad
\parbox[c]{2in}{
\begin{tikzpicture}[scale=.6]
  \node at (-2,0) [shape=circle,minimum size=6pt,draw=black,thick,label=left:$a$] (a) {};
  \node at (0,0) [shape=circle,minimum size=6pt,draw=black,thick,label=below:$c$] (c) {};
  \node at (2,0) [shape=circle,minimum size=6pt,draw=black,thick,label=right:$b$] (b) {};
  \draw[->] (a) to node [above] {$e'$} (c);
  \draw[->] (c) to node [above] {$e''$} (b);
\end{tikzpicture}
}
\end{center}

  \caption{Parallel and series connection.}
  \label{pic2}
\end{figure}

\bigskip 
\bigskip 
\bigskip 
\bigskip 

{\bf General proper case: $r > 0, s> 0$} 

\begin{proposition} (\cite{Gur22}). 
\label{l-ps}
The resistances of these two networks can be
determined, respectively, from formulas

\begin{equation}
\label{eq-ps}
\mu_{a,b}^{-s} = \mu_{e'}^{-s} + \mu_{e''}^{-s}
\;\;\; \mbox{and} \;\;\;
\mu_{a,b}^{s/r} = \mu_{e'}^{s/r} + \mu_{e''}^{s/r}.
\end{equation}
\end{proposition}

The following  properties of the convolution  
$\mu(t) = (\mu_{e'}^t +  \mu_{e''}^t)^{1/t}$ 
are both obvious and well-known: 
\begin{equation}
\label{convolution}
\mu(t) \rightarrow \max(\mu_{e'}, \mu_{e''}),
\;\mbox{as}\;
t \rightarrow + \infty, \;\; 
\mu(t) \rightarrow \min(\mu_{e'}, \mu_{e''}),
\;\mbox{as}\;
t \rightarrow - \infty.
\end{equation}

{\bf Case E, $r=s=1$,}  corresponds to the electrical semiconductor   
networks satisfying the linear Ohm's law; 
$\mu_e$  is the resistance of  $e$ and $\lambda_e = 1/\mu_e$ is its conductance. 
Respectively, 
$\mu_{a,b}$  and $\lambda_{a,b} = 1/\mu_{a,b}$ 
are the resistance and conductance of the two-pole network from $a$ to $b$. 
Applying Proposition \ref{l-ps} for $r=s=1$, we obtain familiar high school formulas  
$$\lambda_{a,b} = \lambda_{e'} + \lambda_{e''} \;\; \text{and}  \;\; 
\mu_{a,b} = \mu_{e'} + \mu_{e''}$$
for the parallel and series networks, respectively.  

\medskip 

{\bf Case H, $r=1/2, s=1,$}  corresponds to the pipeline network 
satisfying quadratic conductance law, typical for hydraulics and gas dynamics; 
$\mu_e, \lambda_e = 1/\mu_e$ and  $\mu_{a,b}, \lambda_{a,b} = 1/\mu_{a,b}$ 
are interpreted similarly. 
Applying Proposition \ref{l-ps} for $r=1/2, s=1$, we obtain 
$$\lambda_{a,b} = \lambda_{e'} + \lambda_{e''} \;\; \text{and}  \;\; 
\mu_{a,b} = (\mu_{e'}^2 + \mu_{e''}^2)^{1/2}$$
for the parallel and series networks, respectively.  

\bigskip 

Now determine $r$ and $s$ for the four limit cases.  
The procedure will be inverse. 
Common sense will help us with two convolutions 
for the parallel and series networks, 
from which $r$  and  $s$  will be derived. 

\medskip

{\bf Case L.}  Every edge  $e$  is a one-way path  
and $\mu_e$  is its length or travel time. 
Furthermore,  $\mu_{a,b}$  is the length of a shortest 
directed path from $a$ to $b$. 
Obviously, 
$$\mu_{a,b} = \min( \mu_{e'}, \mu_{e''}) \;\; \text{and}  \;\; 
\mu_{a,b} = \mu_{e'} + \mu_{e''},$$ 
which results in $r=s \rightarrow \infty$, by Proposition \ref{l-ps}.

\medskip

{\bf Case W.}  Every edge  $e$  is a one-way path  and $\mu_e$  is its width. 
For example, $e$ may be a bridge, tunnel, or canal, 
and  $\lambda_e$  is the maximum weight or height of a car, 
or the tonnage of a ship that may still pass  $e$. 
Furthermore, $\lambda_{a,b}$  is the width of the bottleneck path, 
in other words, it is the maximum ``size'' of an object 
that can pass from  $a$  to  $b$. 
Obviously, 
\[
\lambda_{a,b} = \max( \lambda_{e'}, \lambda_{e''} ) \;\; \text{and} \;\; 
\lambda_{a,b} = \min( \lambda_{e'}, \lambda_{e''} ),
\] 
which results in $r=1, s \rightarrow \infty$, by Proposition \ref{l-ps}.

\medskip

{\bf Case F.}  Every edge  $e$  is a one-way pipe  and $\lambda_e$  is its capacity.  
Furthermore, $\lambda_{a,b}$  is the capacity of the two-pole network, 
that is, the maximal flow from $a$ to  $b$.  
Obviously, 
$$\lambda_{a,b} =  \lambda_{e'} + \lambda_{e''} \;\; \text{and}  \;\; 
\lambda_{a,b} = \min(\lambda_{e'}, \lambda_{e''}),$$ 
which results in $s=1, r \rightarrow 0$, by Proposition \ref{l-ps}. 

\medskip 

These 3 limit cases were considered in \cite{Gur22}. 
Note that calculating $r$ and $s$ as shown above immediately proves the inequality  (\ref{eq-main})  
only in case of the  series-parallel networks. 
The proof for all networks is essentially more complicated (\cite{Gur22}).   

\medskip 

In all cases we obtain quasi-metric spaces, since $s \geq r$;  
in cases $F$ and $W$ we get quasi-ultrametric inequalities, 
since $s/r \rightarrow \infty$. 
Yet, in all limit cases an accurate proof is required, 
which can be found in \cite{Gur22}. 

Note that equality (\ref{eq-main}) 
is equivalent with the condition 
``Every   $(a,b)$-path contains  $c$'' 
for every proper case $r > 0, s > 0$. 
Yet, for the above three limit cases 
the second implies the first one but not vice versa (\cite{Gur22}). 
The same will hold for the limit case B; 
see the next section. 

\bigskip 

Furthermore, cases $E, H, F$ are ``quantum'': 
$\mu_{a,b}$ depends on all  $(a,b)$-paths, 
while cases L and W  are ``classical'': 
$\mu_{a,b}$ is determined by a shortest $(a,b)$-path; 
see \cite{Gur22} for more details. 

\section{Bridge distance. Case B} 
We will define the bridge distance, compute it in polynomial time, prove metric inequality, study when it holds with equality. 
Then, we will show that this distance corresponds to the limit case $r=s \rightarrow 0$. 

\subsection{Definition and main properties} 

Let $G = (V,E)$ be a digraph with the source $a$ and sink $b$. An edge $e \in E$ is called an $(a,b)$-bridge if 
it belongs to each  $(a,b)$-path. In other words, after deleting  $e$, there is no $(a,b)$-path. 
The latter property can be verified in time $O(|V|+|E|)$ by a standard graph traversal algorithm like Breadth- or Depth-First Search.

Note that $e$ may belong to a unique $(a,b)$-path not being an $(a,b)$-bridge.  

\begin{example} 

We have  $\mu(a,b) = 0$, since there are two parallel 2-paths from  $a$  to  $b$. 
Yet, these two paths do not contain  $e = (c,d)$. Thus, $e$  is not an $(a,b)$-bridge, 
although there is a unique $(a,b)$-path through $e$; see Figure \ref{pic3}. 

\begin{figure}[H] 
\centering 
\begin{tikzpicture}[
  ->, 
  shorten >=1pt, 
  auto, 
  node distance=2.5cm,
  thick, 
  main node/.style={circle, fill=white, draw, font=\sffamily\small, thin} 
]

  \node[main node] (a) {a};
  \node[main node] (c) [below right of=a] {c};
  \node[main node] (d) [above right of=c] {d};
  \node[main node] (b) [below right of=d] {b};

  \path
    (a) edge (c)
    (a) edge (d)
    (c) edge (d)
    (c) edge (b)
    (d) edge (b);

\end{tikzpicture}
 \caption{Directed edge $(c,d)$  belongs to a unique  $(a,b)$-path not being an $(a,b)$-bridge}
\label{pic3}
\end{figure}
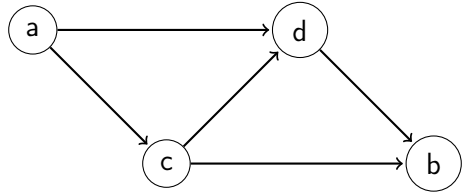

\end{example} 

The next statement is simple but strong. 

\begin{proposition}
\label{order}
All  $(a,b)$-paths pass all  $(a,b)$-bridges in the same order.  
\end{proposition}

\proof 
The first part is straightforward. 

Assume that there are  $k$  $(a,b)$-bridges and that an  $(a,b)$-path  $p'$  passes them  
in the order $e_1, \ldots, e_k$, while another $(a,b)$-path  $p''$ 
passes  $e_1, \ldots e_\ell$, for some $\ell < k$, and then skips $e_{\ell+1}$. 
Then, the latter is not an $(a,b)$-bridge. 
\qed 

This statement can be viewed as a version of the Menger Theorem for directed graphs (\cite{Men927}).

\medskip 

Given a two-pole network, define the bridge distance $\mu_{a,b}$  as follows.
If  $G$  has no $(a,b)$-path then $\mu_{a,b} = \infty$. 
Otherwise,  $\mu_{a,b}$  is the total length of all $(a,b)$-bridges.  
In particular, if  $G$  has $(a,b)$-paths but has no $(a,b)$-bridges 
then $\mu_{a,b} = 0$. 

\begin{proposition}
\label{p3}
The metric inequality 
\[
\mu(a,b)\le \mu(a,c)+\mu(c,b)
\]
holds for all $a,b,c\in V$, and it becomes an equality whenever every
$(a,b)$-path contains $c$.
\end{proposition}

\proof 
If there are no $(a,c)$-paths or no $(c,b)$-paths, then either $\mu(a,c) = \infty$ or $\mu(c,b) = \infty$, and the inequality $\mu(a,b) \leq \mu(a,c) + \mu(c,b)$ holds trivially.

Assume both $(a,c)$-paths and $(c,b)$-paths exist. The concatenation of any $(a,c)$-path and any $(c,b)$-path forms a directed walk from $a$ to $b$, which necessarily contains an $(a,b)$-path. 

Let $e$ be an $(a,b)$-bridge. By definition, $e$ must belong to all $(a,b)$-paths, which includes the paths derived from concatenating any $(a,c)$-path with any $(c,b)$-path. Therefore, $e$ must belong to either all $(a,c)$-paths or all $(c,b)$-paths. In other words, every $(a,b)$-bridge is either an $(a,c)$-bridge or a $(c,b)$-bridge. 

Since the set of $(a,b)$-bridges is a subset of the union of $(a,c)$-bridges and $(c,b)$-bridges, summing their lengths yields the metric inequality:
\[
\mu(a,b) \leq \mu(a,c) + \mu(c,b).
\]

Furthermore, if $c$ belongs to every $(a,b)$-path, then the journey from $a$ to $b$ is strictly partitioned at $c$. Consequently, the set of all $(a,b)$-bridges is exactly the disjoint union of the sets of $(a,c)$-bridges and $(c,b)$-bridges. Therefore, the inequality becomes an equality:
\[
\mu(a,b) = \mu(a,c) + \mu(c,b).
\]
\qed 

Yet, the converse of the last statement of Proposition \ref{p3}  does not hold.

\begin{example} 
Consider the two-pole network in Figure \ref{pic4}, where $\mu_e = 1$ for all $e \in E$.   

\begin{figure}[H] 
\centering 
\begin{tikzpicture}[
  ->, 
  shorten >=1pt, 
  auto, 
  node distance=2.5cm,
  thick, 
  main node/.style={circle, fill=white, draw, font=\sffamily\small, thin} 
]

  \node[main node] (a) {$a$};
  \node[main node] (cp) [right of=a] {$c'$};
  \node[main node] (c) [below right of=cp] {$c$};
  \node[main node] (cpp) [above right of=c] {$c''$};
  \node[main node] (b) [right of=cpp] {$b$};

  \path
    (a) edge (cp)
    (cp) edge (c)
    (c) edge (cpp)
    (cpp) edge (b)
    (cp) edge (cpp);

\end{tikzpicture}
\caption{Equality $\mu(a,b)=\mu(a,c)+\mu(c,b)$ holds, although there is an  $(a,b)$-path avoiding~$c$}
\label{pic4}
\end{figure}
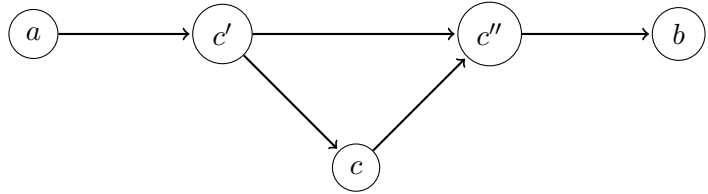

Vertex $c$ does not belong to the $(a,b)$-path $a, c', c'', b$. 
However, 
\[
\mu_{a,b} = \mu_{a,c} + \mu_{c,b} = \mu_{a,c'} + \mu_{c'',b} = 2.
\]  
There are two $(a,b)$-bridges: $e' = (a,c')$ and $e'' = (c'', b)$. 
\end{example}

Analogous examples for the limit cases L, W, and F were given in \cite{Gur22}. 
Recall that for the proper cases, $r>0$ and $s>0$, such examples do not exist. 

\begin{proposition} 
The following four statements are equivalent: 
\begin{itemize} 
\item[(i)] the metric inequality turns into equality for bridge distances;  
\item[(ii)] an $(a,b)$-path passes each $(a,c)$-bridge before each $(c,b)$-bridge; 
\item[(iii)] all $(a,b)$-paths pass each $(a,c)$-bridge before each $(c,b)$-bridge; 
\item[(iv)] for each $(u,v)$-path extendable to an $(a,b)$-path, 
$u$ is between the last $(a,c)$-bridge and $c$,  
while $v$ is between $c$ and the first $(c,b)$-bridge.  
\end{itemize} 
\end{proposition}

\proof 
It follows immediately from Proposition \ref{order}.  

\subsection{Limit case B: $r=s \rightarrow 0$}
Let us show that the bridge distance corresponds to the limit case B.

First, let us consider parallel and series connections in Figure \ref{pic2}. 

For the latter case, we have  
$\mu_{a,b}^{s/r} = \mu_{e'}^{s/r} + \mu_{e''}^{s/r}$ for all $r>0, s>0$, and 
$\mu_{a,b} = \mu_{e'} + \mu_{e''}$ for the bridge distances. 
Thus, we can set $r=s$. 

For the first case, we have  
$\mu_{a,b}^{-s} = \mu_{e'}^{-s} + \mu_{e''}^{-s}$ for all $r>0, s>0$, and 
just $\mu_{a,b} = 0$ for the bridge distances, 
since there are two $(a,b)$-paths and no $(a,b)$-bridges.  
Thus, we can set $s \rightarrow 0$.  

This immediately implies that Case B is described by the limit transition $r=s \rightarrow 0$ 
but only for the series-parallel networks. 
In general, we have to mimic the arguments 
based on the Maxwell minimal dissipation principle and applied for the cases L, W, and F 
in the last section of \cite{Gur22}. 

The Joule-Lenz heat on $e$ is defined by the current $y_e^* \geq 0$
as the integral

\begin{equation}
\label{dissipation-e}
F^*_e(y^*_e) = \int f^{-1}_e (y^*_e) \, d y^*_e  =
\frac {\mu_e^{s/r}} {1 + 1/r} (y_e^*)^{1 + 1/r},
\end{equation}

\noindent
which is a strictly convex function of  $y_e^* \geq 0$.
Furthermore, the total heat dissipated in the network is additive:
\begin{equation}
\label{dissipation}
F^*(y^*) = \sum_{e \in E} F^*_e (y^*_e).
\end{equation}

In a network with the total unit current from  $a$ to  $b$ 
it will be distributed among the edges in such a way 
that the total dissipation $F^*(y^*)$ is minimized; 
see \cite{Gur22}  for the necessary definitions and more details. 

Since $1 + 1/r > 1$ for $r > 0$, we conclude that $F^*$ 
is a strictly convex function of $y^*$ defined in the positive orthant, $y^* \geq 0$. 
Furthermore, in Case~B we have $s=r \rightarrow 0$ and, hence, $1 + 1/r \rightarrow \infty$. 
Obviously, $(1+1/r) F^*_e (y^*_e) \rightarrow 0$ when $y^*_e < 1$, and it equals $\mu_e$ when $y^*_e = 1$.  
Thus, only unit currents $y^*_e = 1$ contribute to (\ref{dissipation-e}), 
while the dissipation $F^*_e(y^*_e)$ vanishes if $y^*_e < 1$. 
But $y^*_e = 1$ only if $e$ is an $(a,b)$-bridge; otherwise, $y^*_e < 1$ in any flow minimizing $F^*(y^*)$. 
Thus, we indeed obtain the bridge distances in Case B.

\section*{Acknowledgements} 
This is an output of a research project (HSE-BR-2025-024) 
implemented as part of the Basic Research Program at HSE University.

\end{document}